\documentclass[12pt, reqno]{amsart}
\usepackage{amsmath, amsfonts}
\newtheorem{remark}{Remark}
\usepackage{amssymb, graphicx}
\usepackage{amscd}
\usepackage{textcomp}
\usepackage{palatino}
\usepackage[colorlinks]{hyperref}
 2

\newtheorem{thm}{Theorem}[section]

\newtheorem{lem}[thm]{Lemma}

\newtheorem{prop}[thm]{Proposition}

\theoremstyle{remark}

\newcommand{\Z}{{\mathbb Z}}

\newcommand{\Q}{{\mathbb Q}}

\newcommand{\N}{{\mathbb N}}

\newcommand{\C}{{\mathbb C}}

\begin{document}
\title[ Harmonic numbers]{Linear Independence of Harmonic Numbers over the field of Algebraic Numbers}
\author[Tapas Chatterjee and Sonika Dhillon]{Tapas Chatterjee\textsuperscript{1} and Sonika Dhillon\textsuperscript{2}}

\address[T. Chatterjee]
      {Department of Mathematics,
      Indian Institute of Technology Ropar,
      Rupnagar, Punjab -140001,
      India.}
\address[S. Dhillon]
       {Department of Mathematics,
      Indian Institute of Technology Ropar,
      Rupnagar, Punjab -140001,
      India.}
\email[Tapas Chatterjee]{tapasc@iitrpr.ac.in}

\email[Sonika Dhillon]{sonika@iitrpr.ac.in}



\maketitle



\begin{abstract}
Let $H_n =\sum\limits_{k=1}^n \frac{1}{k}$ be the $n$-th harmonic number. Euler extended it to complex arguments and defined $H_r$ 
for any complex number $r$ except for the negative integers. In this paper, we give a new proof of the transcendental 
nature of $H_r$ for rational $r$. For some special values of $q>1,$ we give an upper bound for   the number of linearly independent
harmonic numbers $H_{a/q}$ with $ 1 \leq a \leq q$ over the field of algebraic numbers.
Also, for any finite set of odd  primes $J$ with $|J|=n,$ define
$$W_J=\overline{\Q}-\text {span of } \{ H_1, \  H_{a_{j_i}/q_i} | \ 1 \leq a_{j_i} \leq q_i -1, \ 1 \leq j_i \leq q_i-1, \ \ \forall q_i \in J\}.$$
 Finally, we show that $$\text{ dim }_{\overline{\Q}} ~W_J=\sum\limits_{\substack{i=1 \\ q_i \in J}}^n \frac{\phi (q_i )}{2} + 2.$$

\end{abstract}

\section{\bf Introduction}
Let $n$ be a natural number. The $n$-th harmonic number is denoted by $H_n$ and is defined as the sum of the reciprocal of first $n$ natural numbers. Thus
$$ H_n= \sum_{k=1}^n \frac{1}{k}.$$
In 1734, Euler gave an integral representation of the harmonic numbers $H_n$ as
$$H_n =\int_0^1 \frac{1-x^n}{1-x}dx$$
and therefore, for any complex number $r,$  one can define the harmonic number $H_r$ by using the following integral representation
\begin{equation}\label{y}
H_r=\int_0^1\frac{1-x^r}{1-x} dx
\end{equation}
except for  the negative integers where it has a  simple pole. \textcolor{red}{This is easily found from the well-known
series representation of the harmonic numbers given by 
$$H_r=r\sum_{k=1}^\infty{1\over k(r+k)}.$$}

Note that, $H_r$ satisfies the following recurrence relation
$$H_r=H_{r-1}+{1\over r}$$ and the reflection relation 
$$H_r- H_{1-r}={1\over r}+{1\over r-1}-\pi\cot\pi r.$$

It is also known that $H_r$ is closely related to the digamma function for $r$ not being a negative integer by
$$H_r=\psi(r+1) + \gamma$$
 where $\psi(r)$  is the classical digamma function \textcolor{red}{defined as the logarithmic derivative of the classical gamma function $\Gamma(x)$, that is
 $$\psi(x)= \frac{d}{dx}(\ln \Gamma(x))= \frac{\Gamma'(x)}{\Gamma(x)}$$} and $\gamma$  is the Euler-Mascheroni constant.

Also, observe that $\psi(1)=-\gamma$ and more generally, for any rational number $a/q$ with $(a,q)=1$ and $1 \leq a \leq q$, 
we have the following formula of Gauss, discovered in 1813, (for a proof, see \cite{MS1} and \cite{CK})
\textcolor{red}{\begin{equation}\label{0001}
\psi(a/q)= - \gamma -\log(2q) - \frac{\pi}{2} \cot(\pi a/q) + 2\sum_{n=1}^{\lfloor \frac{q-1}{2} \rfloor}  \cos (\frac{2 \pi n a}{q})  \log \sin(\pi n/q).
\end{equation}}

  Hence, for the fractional argument $a/q \text{ with $(a,q)=1,$  }$ the harmonic number satisfies the following identity
\begin{equation}\label{x}
 H_{a/q}= \frac{q}{a}   -\log (2q) - \frac{\pi}{2} \cot(\pi a/q) + 2\sum_{n=1} ^{\lfloor \frac{q-1}{2} \rfloor}    \cos (\frac{2 \pi n a}{q})  \log \sin(\pi n/q).
 \end{equation}
 \smallskip
 In section 3, we will study about  the arithmetic nature of harmonic numbers $H_r$ and their linear independence over the field of algebraic numbers.
 \textcolor{red}{In a forth coming paper \cite{CD}, we study these problems for a more general class of Harmonic numbers. }

\section{\bf Notations and Preliminaries}
In this section, we summarize the known results and notations which we are going to use throughout the paper.
Let $\Q$, $\overline{\Q}$ and $\C$ denote the field of rational numbers, the field of algebraic numbers 
and the field of complex numbers respectively. 

We also need the Baker's theory, which has always been the most important tool for proving the linear 
independence of algebraic  linear combination of logarithm of non-zero algebraic numbers. 
The following proposition due to Baker, (see \cite{AB})  will play a crucial role in proving some of the theorems.

\begin{prop} If $\alpha_1,...,\alpha_n$ are non-zero algebraic numbers such that $\log  \alpha_1,...,\log  \alpha_n$
are linearly independent over the field of rational numbers, then $1, \log  \alpha_1,...,\log \alpha_n$ are linearly independent over the field of algebraic numbers.
\end{prop}

\textcolor{red}{We state some simple applications of Baker's theorem as were done in \cite{CM}, \cite{MM} and \cite{RM} to  resolve the problem of transcendental nature of algebraic linear combination of $\pi$, logarithm of non-zero algebraic numbers and positive real units.}

\begin{prop}
Let $ \alpha_1,..., \alpha_n$ be positive algebraic numbers. If $ c_0,c_1,...,c_n$ are algebraic numbers with $ c_0 \neq 0$, then
$$ c_0 \pi + \sum_{j=1}^n c_j \log \alpha_j$$
 is a transcendental number and hence non-zero.

\end{prop}
\begin{prop}Let  $\alpha_1,\alpha_2,...,\alpha_n$ be \textcolor{red}{positive real units} in a number field of degree $>1.$  
Let $r$  be a positive rational number unequal to 1. If  $c_0,c_1,...,c_n$  are algebraic numbers 
with  $c_0\neq 0,$  and $d$  is an integer, then
$$ c_0 \pi + \sum_{j=1}^n c_j \log \alpha_j + d\log r$$
is a transcendental number and hence non-zero.
\end{prop}
We will use the next proposition due to Chatterjee and Gun to prove some of the theorems in the later section. 
This proposition is motivated from the idea of multiplicative independent  units in a cyclotomic field. Here is the statement of the proposition.

  \begin{prop}
  For any finite set $J$ of   primes in $\N$ with $ p _i \in J$ and $q_i =p_i^{m_i},$ where $m_i \in \N,$ and let $\zeta_{q_i}$ be a primitive $q_i$-th root of unity. Then the numbers 
  $$ 1-\zeta_{q_i},\hspace{.1cm} \frac{1-\zeta_{q_i}^{a_{j_i}}}{1-\zeta_{q_i}}, \hspace{.3cm} \text {where }\hspace{.2cm}  1 <a_{j_i} < q_i/2, \hspace{.1cm} (a_{j_i},q_i)=1 \text{ and }  
  1 < j_i < q_i/2, \ \  \forall \  p_i \in J,$$
  
  are multiplicatively independent. 
  \end{prop}
(For a proof  see \cite{TCS}).  

  \section{\bf \textcolor{red}{Transcendence of Harmonic Numbers}}

 In 2007, Murty and Saradha, (see \cite{MS1})  proved that for $q>1,$  $\psi(a/q) + \gamma$ is transcendental for any $1\leq a\leq q-1$ 
 and hence we can easily deduce that the harmonic number $H_{a/q}$  at the rational arguments is transcendental, whenever $q$ does 
 not divide $a.$ On the other hand, if $q$ divides $a,$   $H_{a/q}$ is a rational  number. Here we are giving  another proof 
 of transcendental nature of harmonic numbers at rational arguments.
  \begin{thm}For $q>1$ and  $a \in \Z,$ $H_{a/q}$ is transcendental, whenever $q$ does not divide $a.$
  \end{thm}
  \begin{proof} We will first consider the case when $1 \leq a \leq q-1.$ Using $\eqref{x},$ we have
  \begin{equation*}
 H_{a/q}- \frac{q}{a} \hspace{.1cm}  = \hspace{.1cm} -\log (2q) - \frac{\pi}{2} \cot(\pi a/q) + 2\sum_{n=1} ^{\lfloor \frac{q-1}{2} \rfloor}    \cos (\frac{2 \pi n a}{q})  \log \sin(\pi n/q).
 \end{equation*}
  Since the right hand side is transcendental for $ \cot(\pi a/q) \neq 0$  by using proposition 2.2, we conclude that the left hand side is transcendental for $a/q \neq 1/2.$ Also, 
  note that for $a/q=1/2, $ $H_{a/q} =2-2 \log 2,$ a transcendental number. 
 Again, observe that  for any $a \in \Z $ and   $n \in \N,$ we have
  \begin{equation*}
  H_{n+a/q}= H_{a/q} + \sum_{k=1}^n \frac{1}{(k+a/q)}
  \end{equation*}
   where $q$ does not divide $a,$ and hence $H_{a/q}$ is transcendental. This completes the proof.
  \end{proof}
   In the next theorem, we will show that the harmonic numbers are distinct in nature.
\begin{thm} For all $q>1,$  \hspace{.06cm} all the numbers in the collection
$$ T=\{ H_{a/q}\hspace{.1cm}  | \hspace{.1cm}  a \in \Z \hspace{.1cm} \text{ and } \hspace{.1cm}  (a,q)=1 \}$$
 are distinct.
\end{thm}
\begin{proof} Note that for any real number $r,$ we have

$$H_r' =\sum_{k=1}^\infty{1\over (r+k)^2} >0$$
except for the negative integers, where $H_r$ is a function of real variable $r$. Thus, $H_r$ is a strictly  increasing function of $r.$ Hence,  elements of  $T$ are distinct.
\end{proof}

 Murty and Saradha also proved the linear independence of $\psi(a/q) + \gamma $ over an algebraic number field where the $q$-th cyclotomic polynomial is irreducible. 
 In our next theorem, we will make some remarks about the linear independence of harmonic numbers $H_{a/q},$ over the field of algebraic numbers.

    First, we prove the following two lemmas that will play a crucial role in proving our theorem for the linear independence of harmonic numbers $H_{a/q}$  over $\overline{\Q}.$

\begin{lem} Let $ n $ be a positive integer such that $n >1 $ and $k $ be an odd positive integer where  $k<2^n.$ Then,\\
 \begin{equation}\label{7}
 \sin(\frac{k \pi}{2^n})= \frac{1}{2} \sqrt{2\pm \sqrt{2\pm...\pm\sqrt2}}
 \end{equation}
 where $\sqrt{\ }$ \  is repeated $  n-1$ times and $\pm $ depends upon the value of $k$.\\
\end{lem}

\begin{proof}
 We will prove the following   lemma by using induction on $n.$
 First suppose that $n=2, \text{ then we have, } 1 \leq k <4$ and it follows trivially. Now we assume that the statement 
 holds true for $m \leq n.$ For the case $m =n+1,$ with $1 \leq k <2^{n+1}$ and $k$ odd,
we use the following  trigonometric formula,
\begin{equation}\label{8}
\sin\frac{k\pi}{2^{n+1}}= \sqrt{\frac{1-\cos\frac{k\pi}{2^n}}{2}}.
\end{equation}
We  first consider the case for $1 \leq k < 2^n.$
Using the identity $\sin^2(x) + \cos^2(x)=1$ and   substituting $\eqref{7}$ in $\eqref{8},$   we have\\
$$ \sin\frac{k\pi}{2^{n+1}}=\sqrt{\frac{1-\frac{1}{2} \sqrt{2\mp \sqrt{2 \pm...\pm\sqrt2}}}{2}}$$
assuming induction hypothesis. Thus,
$$ \sin\frac{k\pi}{2^{n+1}}  = \frac{1}{2} \sqrt{2- \sqrt{2\mp \sqrt{2 \pm...\pm\sqrt2}}}     $$
where $\sqrt{\ }$ \  is repeated $n$ many times. Therefore, it is of the given form.\\
Now for the case  $2^n <k <2^{n+1}$, write $k=\alpha +2^n,$ where $1 \leq \alpha <2^n,$ then
$$ \sin\frac{k\pi}{2^{n+1}}= \cos\frac{\alpha\pi}{2^{n+1}}.$$
Now using the following trigonometric identity for cosine  function
$$ \cos \frac{\alpha \pi}{2^{n+1}}= \sqrt{\frac{1+\cos\frac{\alpha \pi}{2^n}}{2}}$$
and  applying the similar argument for $\alpha$ and $\cos x$ in terms of $\sin x$ as we did in the previous case, we  get the desired result.
\end{proof}

\begin{lem} Let $q=2^n$. Then for any natural number $j$ such that, $\frac{\lfloor \frac{q-1}{2} \rfloor +1}{2} +1 \leq j \leq  \lfloor \frac{q-1}{2} \rfloor$,  
$\log \sin (\frac{j\pi}{q})$ can be written as an algebraic linear combination of $\log \sin (\frac{p\pi}{q}),$ where $1\leq p\leq \frac{\lfloor\frac{q-1}{2}\rfloor+1}{2}.$\\
\end{lem}
\begin{proof}First  note that, using lemma 3.3, for any integer $j$ such that  $\frac{\lfloor \frac{q-1}{2} \rfloor +1}{2} +1  \leq j \leq  \lfloor \frac{q-1}{2} \rfloor,$  we have,
$$\sin(\frac{j\pi}{q})=\frac{1}{2} \sqrt{2\pm \sqrt{2\pm....\pm \sqrt2}}$$
if necessary canceling the multiples of 2 from numerator and denominator.
Then, for  $1\leq (\lfloor\frac{q-1}{2} \rfloor -j+1))   \leq \frac{\lfloor\frac{q-1}{2}\rfloor-1}{2}$ we get 
 $$\sin(\frac{\pi}{q}(\lfloor\frac{q-1}{2} \rfloor -j+1))=\frac{1}{2} \sqrt{2\mp \sqrt{2\pm....\pm \sqrt2}}$$\\
 i.e. there is  a change of sign only at the first position. Also, observe that
 $\log(2\pm \sqrt{2\pm....\pm \sqrt2})$ can be written in terms of $\log (2\mp \sqrt{2\pm....\pm \sqrt2})$ by rationalization as
 $$\log(2\pm \sqrt{2\pm....\pm \sqrt2})=- \log(2\mp \sqrt{2\pm....\pm \sqrt2}) + \log(2\mp \sqrt{2\pm....\pm \sqrt2})$$
 where in the last term of right hand side, the number of iterative square roots is one less than that of the previous term.
Again rationalize the second term (if necessary)  in the above equation  and repeat the process (if required)  until we get $ \log 2$ as the last term. Also note that  
$\log \sin\big(\frac{(\lfloor \frac{q-1}{2} \rfloor +1) \pi}{2q}\big)=-\frac{1}{2} \log 2.$  Hence, we get at most $\frac{\lfloor\frac{q-1}{2}\rfloor+1}{2} $ many 
numbers of the set  $\{ \log \sin (\frac{j\pi}{q}), \ 1\leq j \leq  q-1 \}$ which are linearly independent over $\overline{\Q}.$
 \end{proof}

  \textcolor{red}{Next, we will prove the theorem for the linear independence of harmonic numbers $H_{a/q}$ over algebraic numbers.}
\begin{thm}  For any positive integer $q=2^n,  n\geq 2$ at most $\phi(\phi(q)) + 2$ many   numbers  of the set $\{ H_{a/q}, \ 1 \leq a \leq q \} $  
are linearly independent over ${\overline{\Q}.}$
\end{thm}

\begin{proof}First note that $\lfloor \frac{\lfloor \frac{q-1}{2} \rfloor}{2}\rfloor + 3 =\phi(\phi(q)) + 2$ for $q=2^n.$ 
Let $r$ be an integer with $r>  \lfloor \frac{\lfloor \frac{q-1}{2} \rfloor}{2}\rfloor + 3.$ 
Choose any $r$ many numbers of the set $\{ H_{a/q}, \ 1 \leq a \leq q \} $ say    $\{ H_{\alpha_1/q},..., H_{\alpha_r/q}\},$  where $\alpha_i \in \{1,2,...,q  \}$ 
for $1\leq i \leq r$ and consider  the equation
 \begin{equation*} c_1 H_{\alpha_1/q} +...+c_r H_{\alpha_r/q}=0  \hspace{.35cm}\text{ for } c_i \in \overline{\Q}.\end{equation*}
  \textcolor{red}{Substituting the Gauss formula for $H_{a/q}$  from equation \eqref{x} in the above equation, we get}
 \begin{equation*}
 c_1 \Big{(} \frac{q}{\alpha_1} -\log (2q) - \frac{\pi}{2} \cot(\pi \alpha_1/q) + 2\sum_{k=1} ^{\lfloor \frac{q-1}{2} \rfloor}    \cos (\frac{2 \pi k \alpha_1}{q})  
 \log \sin(\pi k/q) \Big ) +...
 \end{equation*}
\begin{equation*}
 +\ c_r  \Big{(} \frac{q}{\alpha_r} -\log (2q) - \frac{\pi}{2} \cot(\pi \alpha_r/q) + 2\sum_{k=1} ^{\lfloor \frac{q-1}{2} \rfloor}    \cos (\frac{2 \pi k \alpha_r}{q})  
 \log \sin(\pi k/q) \Big )=0.
 \end{equation*}
 Again rewriting the above equation we get,
 \begin{equation*}
 \Big( \frac{c_1q}{\alpha_1} +...+\frac{c_rq}{\alpha_r}\Big)- \frac{\pi}{2}\Big(c_1\cot(\pi \alpha_1/q)+...+c_r\cot(\pi \alpha_r/q)\Big) -\log(2q)\Big(c_1+...+c_r \Big) + 
 \end{equation*}
 \begin{equation}\label{y}
 + \  2\sum_{k=1} ^{\lfloor \frac{q-1}{2} \rfloor}  \log \sin(\pi k/q) \Big( c_1 \cos (\frac{2 \pi k \alpha_1}{q})  +...+ c_r\cos (\frac{2 \pi k \alpha_r}{q}) \Big) =0 .
 \end{equation}
 \textcolor{red}{Now by using proposition 2.2,  we must have 
 $$( \frac{c_1q}{\alpha_1} +...+\frac{c_rq}{\alpha_r} )=0,$$
  $$\Big(c_1\cot(\pi \alpha_1/q)+...+c_r\cot(\pi \alpha_r/q)\Big)=0.$$}
Thus, we get two linear homogeneous equations in the variables $c_i^{~,}s$ where $ 1 \leq i \leq r$. 
Hence,  equation $\eqref{y}$ reduces to
  \begin{equation}\label{ab}
  -\log(2q)\Big(c_1+...+c_r \Big) 
 + 2\sum_{k=1} ^{\lfloor \frac{q-1}{2} \rfloor}  \log \sin(\pi k/q) \Big( c_1 \cos (\frac{2 \pi k \alpha_1}{q})  + \ ...+ c_r\cos (\frac{2 \pi k \alpha_r}{q}) \Big) =0 .
 \end{equation}
  
  Now note that  $\log \sin(\frac{(\lfloor \frac{q-1}{2} \rfloor +1 )\pi}{2q})=-\frac{1}{2} \log 2$. Thus, by lemma 3.4, we get a subset $S$ of
  the set $\{ \log \sin(\frac{\pi}{q}),...,\log \sin(\frac{\lfloor \frac{q-1}{2} \rfloor \pi}{q})\}$  with at most $\frac{\lfloor\frac{q-1}{2}\rfloor+1}{2}$ many numbers   
 which are linearly independent over $\overline{\Q}$. Finally rewriting \eqref{ab} in terms of the linear combination of elements of the subset $S$ and using Baker's theory,
 we can get at most  $\frac{\lfloor\frac{q-1}{2}\rfloor+1}{2}$ many linear homogeneous equations in the variables $c_i^{~,}s$ with algebraic coefficients.
  
  So, altogether we get a linear homogeneous system of  at most $\frac{\lfloor\frac{q-1}{2}\rfloor+1}{2}+2$   many  equations in $r$   variables $c_1,...,c_r,$ where
  $r>  \lfloor \frac{\lfloor \frac{q-1}{2} \rfloor}{2}\rfloor + 3=\frac{\lfloor\frac{q-1}{2}\rfloor+1}{2}+2$ with algebraic coefficients. Thus, there exists a 
  non trivial algebraic solution for $c_i^{~,} s$.
 Hence, the set $ \{H_{\alpha_1/q},..., H_{\alpha_r/q}  \}$ must be linearly dependent over $\overline{\Q}.$ This completes the proof.
\end{proof}

Note that the above theorem deals with the linear independence of harmonic numbers $H_{a/q}$ for $ 1\leq a\leq q ,$  where $q=2^n,$ over $\overline{\Q}$.
In our next theorem, we will prove a more general case. For this we need the following lemmas. 
\begin{lem}
For any finite set $J$ of   primes in $\N$ with $ p _i \in J$ and $q_i =p_i^{m_i},$ where $m_i \in \N,$   the numbers  $ 2\sin \frac{k_{j_i} \pi}{q_i}$ where  
$ 1\leq k_{j_i}<q_i/2 ,$   \ $(k_{j_i} ,q_i)=1 $  and $1\leq j_i < q_i/2,$ for all $p_i \in J, $  are multiplicatively independent.
\end{lem}
\begin{proof}
First note that, $2\sin \frac{k \pi}{q}=\frac{\zeta_{q}^{k/2}-\zeta_{q}^{-k/2}}{i}$ where $\zeta_{q} = e^{\frac{2 \pi i}{q}}$ is a primitive $q$-th root of unity.
\textcolor{red}{Also, observe that the numbers   $ 2\sin \frac{k_{j_i} \pi}{q_i}= |1-\zeta_{q_i}^{k_{j_i}}|$ where  $ 1\leq k_{j_i}<q_i/2 ,$   \ $(k_{j_i} ,q_i)=1 $  and $1\leq j_i < q_i/2,$ for all 
$i$ such that $p_i\in J,$    are multiplicative independent if and only if    $ 1-\zeta_{q_i}, \frac{1- \zeta_{q_i}^{k_{j_i}}}{1-\zeta_{q_i}}$ where 
$1<k_{j_i} <q_i/2,$ \  $ (k_{j_i}, q_i)=1$ and $1\leq j_i < q_i/2, $ \  $ p_i \in J$  are multiplicative independent.} Now the later part  follows from proposition 2.4.
\end{proof}

Note that from the  lemma 3.6,  we can deduce that for any finite set $J$ of odd primes with $p_i \in J$  and $q_i= p_i ^{m_i},$ where $m_i \in \N,$  the numbers
$$\{ \log 2,\   \log \sin \frac{k_{j_i}\pi}{q_i}\  |  \ \ 1\leq k_{j_i} <q_i/2, \  (k_{j_i}, q_i)=1  \text{ and } 1 \leq j_i <q_i/2, \  \  \forall  p_i \in J \}    $$
are linearly independent over $\Q$ and hence over $\overline{\Q} $  by using Baker's theorem.

\textcolor{red}{
\begin{lem}
For any odd number $q,$ let $\alpha=\frac{q-1}{2}.$ Then,
$$\sin (\frac{\pi}{q}) \sin( \frac{2 \pi}{q})...\sin( \frac{\alpha \pi}{q})= \frac{\sqrt q}{2^{\alpha}}.$$
\end{lem}
\begin{proof}
To prove lemma 3.7, we need the following multiple-angle formula for sine function that is for any odd positive integer $q$ and real number $x$, we have
$$\sin (qx)=2^{q-1} \prod_{k=0}^{q-1} \sin (x+ \frac{k\pi}{q}).$$
Dividing the above equation by $\sin x$ and letting $x \rightarrow 0,$ we  get 
$$q=2^{q-1} \prod_{k=1}^{q-1} \sin ( \frac{k\pi}{q}).$$
Since $q$ is an odd integer , we will get the desired result.
\end{proof}}

Note that from the above lemma,  for any odd prime $q$ with $\alpha = \frac{q-1}{2},$ we have 
\begin{equation}\label{v}
\log q=2\Big(\log \sin(\frac{\pi}{q}) +... + \log  \sin(\frac{\alpha\pi}{q}) + \alpha \log 2\Big).
\end{equation}

We also need the following technical lemma to prove our theorem. Here is the statement of the lemma.

\begin{lem} Let $q $ be an odd integer such that $ 3 \nmid q$. Then for any 
 natural number $j$ such that, $\frac{q+1}{2}  \leq j \leq q-1,$  \ $\log \sin (\frac{j\pi}{2q})$ can be written as an algebraic linear combination of 
 $\log \sin (\frac{k\pi}{2q})$ and $\log 2$,  for $1\leq k\leq \frac{q-1}{2} .$
\end{lem}
\begin{proof}Consider $S= \{1,2,...,\frac{q-1}{2} \}$ and $S'=\{ \frac{q+1}{2} ,...,q-1 \}.$
Let $j \in S'$  and put  $j_1=j.$ Choose an element $r$ such that $r+j=q.$ Since $\frac{q+1}{2}  \leq j \leq q-1,$ hence we get  $1\leq r\leq \frac{q-1}{2} .$
Now consider $$\sin (\frac{j \pi}{2q}) \sin (\frac{r \pi}{2q}) = \frac{1}{2} \big(\cos (\frac{(j-r) \pi}{2q}) - \cos (\frac{(j+r) \pi}{2q})\big)$$
$$=\frac{1}{2} \sin \big(\frac{(q-(j-r)) \pi}{2q}\big),$$
as $\cos (\frac{(j+r) \pi}{2q}) = 0.$
Thus, we get
$$\log\sin (\frac{j \pi}{2q})=   -\log \sin (\frac{r \pi}{2q}) +   \log\sin \big(\frac{(q-(j-r)) \pi}{2q}\big) - \log 2 .$$

Now put $ j_2 = q-(j-r)$ and if $j_2 \in S,$ we are done, else find an $r'$ such that  $j_2 + r' =q$ and repeat the same process for $j_2$ as we did for $j_1.$ 
Again we will get a $j_3$ and repeating this process we will get a sequence $j_n.$  Note that, if $j_k \in S$  for  some $k \in \N,$ then our lemma holds. Otherwise, 
we will get a sequence $j_n \in  S'$ for all $n \geq 1.$

 Assume that, there exist an integer $j$ such that $j_n \in  S'$ for all $n \geq 1.$  It is very easy to observe that  $j_{n+1} = 2q -2j_n$ 
 for all $n\geq1.$ Also by using induction one can prove that $$j_{n+1} =2q\big( \frac{1-(-2)^{n}}{3} \big) + (-2)^{n}j_1$$
 where $j_1=j.$  Since $j_{n+1} \in S',$ hence
 $$j_{n+1} \geq \frac{q+1}{2}   \hspace{.2cm} \text{for all } n \geq 1, $$
 from where we get 
 $$2q\big( \frac{1-(-2)^{n}}{3} \big) + (-2)^{n}j \geq \frac{q+1}{2} $$  
 and 
 \begin{equation}\label{g}
  (-1)^{n}j \geq  \frac{-q+3 + 4q(-2)^n}{ 2^{n+1} 3}.
  \end{equation} 
  
 Also observe that initially, we have $\frac{q+1}{2}  \leq j \leq q-1.$ Let $a_0= \frac{q+1}{2}$ and $b_0 =q-1.$ Now take $n=1$ in $\eqref{g}$, 
 then we will get $j \leq \frac{3q-1}{2^2} $  and hence $\frac{q+1}{2} \leq j \leq \frac{3q-1}{2^2}.$ Again put  $a_1 =\frac{q+1}{2} $ and $b_1 =\frac{3q-1}{2^2}.$ 
 Continuing this process for any n, we will get $a_n$ and $b_n$ such that $a_n \leq j \leq b_n.$  A quick calculation will give
 $$a_n =\frac{-q+3 + (-2)^{n} 4q}{ 2^{n+1} 3}, \text{ where } n \text{ is even }$$
 and 
 $a_{n+1}=a_n.$ Similarly, 
  $$b_n =-\Big(\frac{-q+3 + (-2)^{n} 4q}{ 2^{n+1} 3}\Big), \text{ for } n \text{ odd }$$
  and $b_{n+1}= b_n.$
 Also note that for any $n,$  we have $b_n -a_n= \frac{q-3}{2^{n+1}}.$ Choose an integer $n$, say $n_0$, 
such that $\frac{q-3}{2^{n_0+1}}<1$.  Thus, there exist at most one integer $j\in S'$ such that $a_{n_0}\le j \le b_{n_0}$.
Therefore, one more iteration will give $2q-2j=j$, which implies that $j={2q \over 3}$. This is a contradiction as $j$ is an integer and
$3\nmid q$ . This completes the proof.
\end{proof}
Now before proving our theorem, we will discuss another important  lemma that will play a crucial role in proving our theorem
(See \cite{bass}, \cite{CMP} and \cite{ennola}).

 \begin{lem}
  \textcolor{red}{For any  $ m = q^n$, where $q$ is an odd prime and $n>1,$  $ \log \sin \frac{k \pi }{m}$ where $ (k, m)>1$  }
 and $ 1 \leq k \leq m-1$ can be written as a linear combination of $ \log 2 \text { and } \log \sin \frac{r \pi}{m}$ where $ (r,m)=1$ and $ 1 \leq r \leq m-1.$ 
 \end{lem}

 \begin{proof}

 Let $\zeta_m = e^{2\pi i /m},$ be the primitive $m$-th root of unity. Consider the numbers $1-\zeta_m^x$ for $x= 1,2,3,...,m-1.$ Then for any divisor $b>1$ of $m$ , we have
 $$\log |1-\zeta_m^{(m/b)x}|= \sum\limits_{\substack{u=1 \\ u \equiv x \text{  mod }  b}}^{m-1}\log |1-\zeta_m^u | .$$
 Also for any $ 1 \leq k \leq m-1,$
 $$\sin \frac{k \pi}{m}= \frac{e^{-i k \pi/m} (\zeta_m^k-1)}{2i}.$$
  Thus  for any divisor $b$ of $m$, 
  $$ \sin \frac{bx \pi }{m}= \ 2^{b-1}  \prod\limits_{\substack{u=1 \\u \equiv x \text{ mod } m/b}}^{m-1} \sin \frac{ u \pi}{m}.$$
Hence, for any divisor $b$ of $m$ and  $ x = 1 ,2,..., m-1, $ we have 
\begin{equation}\label{at}
\log \sin\frac{bx \pi}{m}= (b-1) \log 2 \ +  \sum\limits_{\substack{u=1 \\ u \equiv x \text { mod } m/b}}^{m-1} \log \sin \frac{u \pi }{m}.
\end{equation}

  Let $1 \leq k \leq m-1$ be such that $(k,m)>1.$
 Suppose that $k=q^r t$ where $ 1 \leq r < n$ and  $(t,q)=1.$ 
 Then substituting the value of   $b=q^r$ and $x=t$ in \eqref{at}, we get
 \begin{equation}\label{aaz}
 \log \sin\frac{k \pi}{m}= (q^r-1) \log 2 \ +  \sum\limits_{\substack{u=1 \\ u \equiv t \text { mod } m/b}}^{m-1} \log \sin \frac{u \pi }{m}.
 \end{equation}
 From the above equation,
 we have $ u= t + q^{n-r}s$ where $ s \in \{0,1,2,...,q^r-1   \}.$ Since $(t,q)=1,$ hence $(u,m)=1.$
\end{proof}

Now by using the previous lemmas, we will make an important remark.

\begin{remark}
 Let $J$ be any finite set of odd primes in $\N$ with $|J|=n$ and $3\notin J.$  Consider the set
$$\{ \log 2, \hspace{.06cm} \log q_i, \  \log \sin (\frac{k_{j_i}\pi}{2q_i^{m_i}})\hspace{.1cm} |  \hspace{.14cm} 1 \leq k_{j_i} \leq q_i^{m_i} -1 
\hspace{.3cm}\text{and}\hspace{.3cm}  1 \leq j_i \leq q_i^{m_i}-1, \ \ \forall \ q_i \in J, \  m_i \in \N\}.$$
Then by using lemma 3.6, 3.7, 3.8 and lemma  3.9  out of the above set, the subset
$$\{ \log 2, \ \log \sin (\frac{k_{j_i}\pi}{q_i^{m_i}})\hspace{.1cm} |  \hspace{.14cm} 1 \leq k_{j_i} \leq \frac{q_i^{m_i}-1}{2} \hspace{.3cm}\text{and}
\hspace{.3cm}  1 \leq j_i \leq \frac{q_i^{m_i}-1}{2},\    (k_{j_i},q_i^{m_i})=1\}$$
where $ q_i \in J$ and  $ m_i \in \N $,
is a maximal linearly independent subset over $\Q$ and hence by using  Baker's theorem, the numbers
$$\{ 1, \log 2, \ \log \sin (\frac{k_{j_i}\pi}{q_i^{m_i}})\hspace{.1cm} |  \hspace{.14cm} 1 \leq k_{j_i} \leq \frac{q_i^{m_i}-1}{2} \hspace{.3cm}\text{and}
\hspace{.3cm}  1 \leq j_i \leq \frac{q_i^{m_i}-1}{2},\    (k_{j_i},q_i^{m_i})=1\}$$
are linearly independent over $\overline{\Q}$.
\end{remark}

Now we are in a position to state our theorem which establish the upper bound of  linear independence of harmonic numbers over the field of algebraic numbers  for some special cases.

\begin{thm}

For any finite set $J$ of odd primes  with $|J|=n$ and $ 3 \notin J$,  at most
 $$\sum\limits_{\substack{i=1 \\ q_i \in J}}^n \frac{\phi (q_i^{m_i} )}{2} + 3 $$
 many numbers of the set 
 $$\{H_1,\ H_{(a_{j_i}/2^{e_i}q_i^{m_i})} \ | \   1\leq a_{j_i}  \leq 2^{e_i}q_i^{m_i}-1  \text{ and } 1 \leq j_i \leq 2^{e_i}q_i^{ m_i }-1, \ \    \forall  \ \ q_i \in J,\  m_i \in \N\}$$   
 where $e_i \in \{0,1\}$,  are linearly independent over algebraic numbers.

\end{thm}
\begin{proof}

First note that for any $1 \leq a \leq 2q-1,$
 $$H_{a/2q}= \frac{2q}{a}   -\log (4q) - \frac{\pi}{2} \cot(\pi a/2q) + 2\sum_{k=1} ^{\lfloor \frac{2q-1}{2} \rfloor}    \cos (\frac{ \pi k a}{q})  \log \sin(\frac{\pi k}{2q}) $$
Choose an $r \in \N$ with $r>\sum\limits_{\substack{i=1 \\ q_i \in J}}^n \frac{\phi (q_i^{m_i} )}{2}  +3 .$  Let $T$ be any  subset of the  set
$$\{ H_1,\hspace{.08cm}  H_{(a_{j_i}/2^{e_i}q_i^{m_i})} \ | \hspace{.12cm}  1\leq a_{j_i}  \leq 2^{e_i}q_i^{m_i}-1 \hspace{.16cm} \text{ and } 1\leq j_i \leq 2^{e_i}q_i^{m_i}-1 \ \text{ for all }  q_i \in J\},$$ 
 where $e_i \in \{ 0,1\}$, containing $r$ many elements. Consider the  $\overline{\Q}$- linear combination of the set $T$  
 \begin{equation*}
 \sum\limits_{\substack{s \\ q_s \in J }}\sum_{t_s}  c_{st_s} H_{(t_s/2^{e_s}q_s^{m_s})}=0
 \end{equation*}
 where $1 \leq t_s \leq 2^{e_s}q_s^{m_s}$ and  the sum runs over all  $s,t_s$  such that $H_{(t_s/2^{e_s}q_s^{m_s})} \in T.$ \   Substituting the value of $H_{(t_s/2^{e_s}q_s^{m_s})}$ in 
 the above equation, we get
 \vspace{.3cm}
$$ \sum\limits_{\substack{s \\ q_s \in J }}\sum_{t_s}  c_{st_s} \Big(\frac{2^{e_s}q_s^{m_s}}{t_s}   -m_s\log (q_s) - (e_s+1) \log 2- \frac{\pi}{2} \cot(\pi t_s/2^{e_s}q_s^{m_s})\Big) $$
 \begin{equation}\label{baa}
 + \Bigg( 2\sum\limits_{\substack{s \\ q_s \in J }}\sum_{t_s}  c_{st_s}\bigg( \sum_{k=1} ^{\lfloor \frac{2^{e_s}q_s^{m_s}-1}{2} \rfloor}    
\Big(  \cos (\frac{ 2\pi k t_s}{2^{e_s}q_s^{m_s}})  \log \sin(\frac{\pi k}{2^{e_s}q_s^{m_s}})  \Big) \bigg) \Bigg)=0.
 \end{equation}
Now, by using proposition 2.2  in  the above equation, we get two linear homogeneous equations in the variables $c_{st_s}$ :
$$
 \sum\limits_{\substack{s \\ q_s \in J }}\sum\limits_{t_s}  \big(c_{st_s} \big(\frac{2^{e_s}q_s^{m_s}}{t_s}\big)\big) =0  $$
 and
 $$
 \sum\limits_{\substack{s \\ q_s \in J }}\sum\limits_{t_s} \Big( c_{st_s}   \cot(\pi t_s/2^{e_s}q_s^{m_s})\Big)=0.$$
Thus,  equation \eqref{baa} reduces to
$$  \sum\limits_{\substack{s \\ q_s \in J }}\sum_{t_s}  c_{st_s} \Bigg(   -m_s\log (q_s) - (e_s+1) \log 2  \Bigg) \ +  $$
 \begin{equation}\label{baaa}
+ 2 \sum\limits_{\substack{s \\ q_s \in J }}\sum_{t_s}  c_{st_s} \Bigg(    \sum_{k=1} ^{\lfloor \frac{2^{e_s}q_s^{m_s}-1}{2} \rfloor}    
\Big(  \cos (\frac{ 2\pi k t_s}{2^{e_s}q_s^{m_s}})  \log \sin(\frac{\pi k}{2^{e_s}q_s^{m_s}})  \Big)\Bigg)=0.
 \end{equation}
 Observe that for any odd prime $q_s$ by using \eqref{v}, we have
\vspace{.3cm}
$$\log q_s=2\Big(\log \sin(\frac{\pi}{q_s}) +... + \log  \sin(\frac{\alpha\pi}{q_s}) + \alpha_s \log 2\Big).
$$
where $\alpha_s=\frac{q_s-1}{2}$. Substitute the value of $\log(q_s)$ in \eqref{baaa}.
 Now we want to find a maximal linearly independent subset of the set
 \vspace{.2cm}
$$  H= \{ \log 2,  \   \log \sin(\frac{\pi k}{2^{e_s}q_s^{m_s}}) \ | \  1 \leq k \leq {\lfloor \frac{2^{e_s}q_s^{m_s}-1}{2} \rfloor}      \} 
 $$
  where $s$ is such that $H_{t_s/2^{e_s}q_s^{m_s}}\in T$.
 \textcolor{red}{ 
First consider the case when $e_s=0$. Then  out of  the set
\vspace{.3cm}
$$
 H_1= \{  \  \log 2,  \   \log \sin(\frac{\pi k}{q_s^{m_s}}) \ | \  1 \leq k \leq {\lfloor \frac{q_s^{m_s}-1}{2} \rfloor}      \} 
 $$
 the subset
 $$S=  \{    \log 2,  \   \log \sin(\frac{\pi k}{q_s^{m_s}}) \ | \  1 \leq k \leq { \frac{q_s^{m_s}-1}{2} }  ,\  (k, q_s)=1    \}$$
  is a maximal linearly independent subset by using lemma 3.6, 3.7 and 3.9 having at most $\sum\limits_{\substack{i=1 \\ q_i \in J}}^n \frac{\phi (q_i^{m_i} )}{2} + 1$
many elements. Now consider  the case when $e_s=1$. Then by using lemma 3.8, the set
$$T_1=\left \{\log \sin(\frac{\pi k}{2q_s^{m_s}}) , \text{ where }\  \frac{q_s^{m_s}+1}{2} \leq k \leq q_s^{m_s}-1\right \},$$
\textcolor{red}{can be written as an algebraic linear combination of the set}
$$T_2=\left \{ \log2, \  \log \sin(\frac{\pi k}{2q_s^{m_s}}) , \text{ where }\  1 \leq k \leq \frac{q_s^{m_s}-1}{2}\right \} .  $$
Thus, for $e_s=1$, out of the set
\vspace{.3cm}
$$
 H_2=  \{  \  \log 2,  \   \log \sin(\frac{\pi k}{2q_s^{m_s}}) \ | \  1 \leq k \leq {\lfloor \frac{2q_s^{m_s}-1}{2} \rfloor}      \} 
 $$
 the subset
 \vspace{.3cm}
 $$S=  \{    \log 2,  \   \log \sin(\frac{\pi k}{q_s^{m_s}}) \ | \  1 \leq k \leq { \frac{q_s^{m_s}-1}{2} }  ,\  (k, q_s)=1    \}$$ 
is a maximal linearly independent set over $\overline{\Q}$, since $H_1$ is contained in $H_2$ and for $H_1$ the subset $S$ is a maximal linearly independent subset.
 Thus, in both the cases the subset 
  $$S=  \{    \log 2,  \   \log \sin(\frac{\pi k}{q_s^{m_s}}) \ | \  1 \leq k \leq { \frac{q_s^{m_s}-1}{2} }  ,\  (k, q_s)=1    \}$$
  is a maximal linearly independent subset of the set
 \vspace{.3cm}
 $$
 H=  \{    \log 2,  \   \log \sin(\frac{\pi k}{2^{e_s}q_s^{m_s}}) \ | \  1 \leq k \leq {\lfloor \frac{2^{e_s}q_s^{m_s}-1}{2} \rfloor}      \} 
 $$
 where $s$ is such that $H_{t_s/2^{e_s}q_s^{m_s}}\in T$ for some $t_s$. } 
Rewriting \eqref{baaa} in terms of elements of the set $S$, we get an equation of the form

\begin{equation}\label{laaa}
\sum\limits_{\substack{s \\ q_s \in J }}\sum_{t_s}  c_{st_s} \Bigg(    \alpha_s \log 2+ 2\sum\limits_{\substack{k=1,\\(k,q_s)=1}} ^\frac{q_s^{m_s}-1}{2}
\Big(  \beta_k^s  \log \sin(\frac{\pi k}{q_s^{m_s}})  \Big)\Bigg) =0
\end{equation}
where $\alpha_s , \  \beta_k^s \in \overline{\Q}.$
Since $S$ is a linearly independent set over $\overline{\Q}$, therefore using this in \eqref{laaa}, we get
$$\sum\limits_{\substack{s \\ q_s \in J }}\sum_{t_s}  c_{st_s}     \alpha_s =0,$$
$$\sum\limits_{\substack{s \\ q_s \in J }}\Big(\sum_{t_s}  c_{st_s} \big(\beta_k^s   \big)\Big)=0,$$
for all $ 1 \leq k \leq \frac{q_s^{m_s}-1}{2}$ with $(k,q_s)=1$ and $s,t_s$ are  such that $H_{t_s/2^{e_s}q_s^{m_s}} \in T.$
So we will get at most $\sum\limits_{\substack{i=1 \\ q_i \in J}}^n \frac{\phi (q_i^{m_i} )}{2} + 1 $ many linear homogeneous equations in the variables $c_{st_s}$.
Thus, altogether we will get a linear homogeneous system of at most 
$\sum\limits_{\substack{i=1 \\ q_i \in J}}^n \frac{\phi (q_i^{m_i} )}{2} + 3 $
many equations with algebraic coefficients  in  $r> \sum\limits_{\substack{i=1 \\ q_i \in J}}^n \frac{\phi (q_i^{m_i} )}{2} + 3 $ many variables $c_{st_s}.$ 
Thus, we will get a non-trivial algebraic solution for $c_{st_s}$. This completes the proof.
\end{proof}

 \section{\bf \textcolor{red}{ Dimension of the space generated by Harmonic Numbers}}

So far we have investigated about the upper bound for the number of  linearly independent harmonic numbers over the field of algebraic numbers. 
In this section, we will study the linear spaces generated by harmonic numbers over $\overline{\Q}$ and their dimensions. Our next theorem will give the dimension 
of the vector space generated by harmonic numbers over $\overline{\Q}$. But before that we will make an  important  remark .
\begin{remark}
 Note that for any odd prime $q,$ with $q \neq 3,$ we have  from   theorem 3.10 that at most $\frac{\phi(q)}{2} +3$ many numbers of the set 
 $\{H_{a/q} \ | \ 1 \leq a \leq q\}$ can be linearly independent over $\overline{\Q}$.
Choose any $r= \frac{\phi(q)}{2} +3$  many elements of the set $\{H_{a/q} \ | \ 1 \leq a \leq q\}$ say, $H_{\alpha_1/q},...,H_{\alpha_r/q},$ 
where $1\leq \alpha_i \leq q$ \  for $1 \leq i \leq r,$ such that
$$c_1 H_{\alpha_1/q} + ... + c_rH_{\alpha_r/q}=0,$$   
where $c_i^{~,}s \in \overline{\Q}.$ \textcolor{red}{Using the Gauss formula  in equation \eqref{0001} for  $H_{a/q},$  the above equation becomes }

 \begin{equation*}
 \Big( \frac{c_1q}{\alpha_1} +...+\frac{c_rq}{\alpha_r}\Big)- \frac{\pi}{2}\Big(c_1\cot(\pi \alpha_1/q)+...+c_r\cot(\pi \alpha_r/q)\Big) -\log(2q)\Big(c_1+...+c_r \Big) + 
 \end{equation*}
 \begin{equation*}
 + \  2\sum_{n=1} ^{\lfloor \frac{q-1}{2} \rfloor}  \log \sin(\pi n/q) \Big( c_1 \cos (\frac{2 \pi n \alpha_1}{q})  +...+ c_r\cos (\frac{2 \pi n \alpha_r}{q}) \Big) =0 .
 \end{equation*}
 By using proposition 2.2, the above equation reduces to
 \begin{equation*}
  -\log(2q)\Big(c_1+...+c_r \Big) 
 + 2\sum_{n=1} ^{\lfloor \frac{q-1}{2} \rfloor}  \log \sin(\pi n/q) \Big( c_1 \cos (\frac{2 \pi n \alpha_1}{q})  +...+ c_r\cos (\frac{2 \pi n \alpha_r}{q}) \Big) =0 .
 \end{equation*}
Now by using Baker's theorem and lemma 3.6, 3.7,  the numbers
$$ \log 2q, \ \log \sin (\frac{n\pi}{q}),\  \text{ where } 1 \leq n \leq \frac{q-1}{2}, $$
are linearly independent over $\overline{\Q}.$
Thus, using the similar idea as we applied in theorem 3.10,  we will get a linear homogeneous system of equations in $r$ many equations and  $r$ variables 
$c_i^{~,}s.$ Now let $\alpha=\frac{q-1}{2},$ then  this system of linear homogeneous equation is equivalent to a matrix equation of the form  $Ax=0,$ where   
$A$ is  an $ r \times r$ matrix and  $x=[c_1,...,c_r]^T,$ which is given by

\[
\begin{bmatrix}
    \frac{q}{\alpha_1}   \quad     & \frac{q}{\alpha_1}~~   & \dots & ~~ \frac{q}{\alpha_r} \\ \\
    
    \cot(\frac{\pi \alpha_1}{q})\quad     &  \cot(\frac{\pi \alpha_2}{q})~~ & \dots & \cot(\frac{\pi \alpha_r}{q})  \\ \\
    
    1 & 1 & \dots & 1 \\ \\
     \cos(\frac{2\pi \alpha_1}{q}) \quad    &  \cos(\frac{2\pi \alpha_2}{q}) & \dots & \cos(\frac{2\pi \alpha_r}{q})  \\ \\
     
      \cos(\frac{4\pi \alpha_1}{q})  \quad    &  \cos(\frac{4\pi \alpha_2}{q}) & \dots & \cos(\frac{4\pi \alpha_r}{q})  \\ \\
      \vdots & \vdots & \vdots & \vdots \\ \\
    
     \cos(\frac{2\pi \alpha \alpha_1}{q})     &  \cos(\frac{2\pi \alpha\alpha_2}{q}) & \dots & \cos(\frac{2\pi \alpha \alpha_r}{q})  
\end{bmatrix}
\begin{bmatrix}{}
    ~c_1~ \\ \\
    c_2 \\ \\
    c_3 \\ \\
    \vdots \\ \\
     \vdots \\ \\
     \vdots \\ \\
    ~c_r~
\end{bmatrix}
=
\begin{bmatrix}{}
    ~0 ~\\ \\
    ~0~ \\ \\
    ~0 ~\\ \\
    \vdots \\ \\
     \vdots \\ \\
     \vdots \\ \\
   ~ 0~
\end{bmatrix}
\]

Observe that for any $1 \leq \alpha_j \leq \alpha,$ where $\alpha=\frac{q-1}{2},$ we have $ \cos \frac{2t\alpha_j \pi} {q}=\cos \frac{2t(q-\alpha_j)\pi}{q}$, 
for all $1  \leq t \leq \alpha. $ Hence for each $ 1 \leq t \leq \alpha,$ by using pigeon-hole principle,  out of the collection 
$\{ \cos (\frac{2 t \alpha_j \pi}{q}) | \hspace{.2cm} 1\leq j \leq r, \hspace{.2cm} 1\leq \alpha_j \leq q-1\}$ at least three numbers must be repeating twice.  
From this we can deduce that in the above matrix $A_{ij},$ in   each row $i$ with $4 \leq i \leq \alpha+3,$ at least three entries are repeating. Now suppose 
that for some $1 \leq t_1 \leq \alpha,$ we have $ \cos \frac{2t_1\alpha_i \pi} {q}=\cos \frac{2t_1\alpha_j\pi}{q}$ for $ 1 \leq i ,j \leq r,$ then
$ \alpha_j = q - \alpha_i.$ Hence,  $ \cos \frac{2t\alpha_i \pi} {q}=\cos \frac{2t\alpha_j\pi}{q}$ for all $ 1 \leq t \leq \alpha.$ Thus, we can conclude that
if $ A_{ij_1}= A_{ij_2}$ for some $i$ with $ 4 \leq i \leq \alpha,$ then $ A_{ij_1}= A_{ij_2}$ for all $ 4 \leq 
i \leq \alpha .$
 Now by performing some elementary column operation on the  above matrix, it is easy to get
$A_{ij}=0,$ for $3\leq i \leq r$ and $r-2 \leq j \leq r.$ From here we can deduce  that the determinant of $A$ is zero. Thus we will get a non-trivial solution 
for $c_i^{~,}s.$  Hence, we can conclude that at most $\frac{\phi(q)}{2} +2$ many numbers of the set $\{H_{a/q} | 1 \leq a \leq q\}$ can be linearly independent 
over $\overline{\Q}.$

\end{remark}
Now we are in a state to prove our next theorem that gives the exact number of elements which are linearly independent over $\overline{\Q}.$

\begin{thm}
For any   prime  $q,$ define $W_q=\overline{\Q}$\ -\ span of  $\{H_{a/q} \ | \ 1\leq a \leq q\}.$ Then, $$\text{ dim }_{\overline{\Q}} ~W_q=\frac{\phi(q)}{2} + 2.$$
\end{thm}
\begin{proof}
For $q\in \{2, 3\},$ the theorem follows easily. Now for the case when $q\notin \{2, 3\},$ let $\alpha=\frac{q-1}{2}.$ Also, observe that from  remark 2,  
$\text{ dim } W_q \leq \frac{\phi(q)}{2}+ 2.$

Consider the subset  $\{ H_{1/q}, \ H_{2/q},...\ ,H_{(\alpha+1)/q},\  H_1\}$ containing $\frac{\phi (q)}{2}+ 2$ many numbers.  To show that  dim 
$W_q=\frac{\phi(q)}{2}+ 2,$ it is sufficient to show that the above subset  is linearly independent over $\overline{\Q}.$
Consider the  $\overline{\Q}$- linear combination of the form 
\begin{equation}\label{h} 
 c_1H_{1/q} + ... + c_{\alpha+1} H_{(\alpha+1) /q} + c_{\alpha+2}H_1=0, \text{ where } c_i \in \overline{\Q}.
 \end{equation}
 Now substituting the value of $ H_{a/q},$ in $\eqref{h}$ we get
\begin{equation*}
\Big( c_1 q \ +\ ...\ + \  c_{\alpha+1} \frac{q}{(\alpha+1)} \ + \ c_{\alpha+2}\Big) -\log (2q)\big( c_1 \ +...+ \  c_{\alpha+1}\big) - 
\frac{\pi}{2}\Big(c_1\cot \frac{\pi}{q}\ +... + \ c_{\alpha+1} \cot \frac{(\alpha+1)\pi}{q}\Big)
\end{equation*}
\begin{equation}\label{j}
 + \ 2\sum_{l=1} ^{ \alpha } \sum_{k=1}^{\alpha+1}    \log \sin(l\pi /q) \big(c_k\cos (\frac{2 lk\pi }{q})\big) =0.
\end{equation}
\textcolor{red}{
Again by using proposition 2.2, we must have 
$$ \Big( c_1 q \ +\ ...\ + \  c_{\alpha+1} \frac{q}{(\alpha+1)} \Big) =0,$$
$$\Big(c_1\cot \frac{\pi}{q}\ +... + \ c_{\alpha+1} \cot \frac{(\alpha+1)\pi}{q}\Big)=0.$$
Thus, equation \eqref{j} reduces to
$$ -\log (2q)\big( c_1 \ +...+ \  c_{\alpha+1}\big) 
 + \ 2\sum_{l=1} ^{ \alpha } \sum_{k=1}^{\alpha+1}    \log \sin(l\pi /q) \big(c_k\cos (\frac{2 lk\pi }{q})\big) =0.$$
Now by using remark 1 in the above equation, we must have
$$ \big( c_1 \ +...+ \  c_{\alpha+1}\big) =0 \ \  \text{ and }
 \  \  \sum_{k=1}^{\alpha+1}   \big(c_k\cos (\frac{2 lk\pi }{q})\big) =0,$$
   for all $ 1 \leq l \leq \alpha.$
Thus, altogether we get a linear homogeneous system of $\alpha+3$ many equations in 
$\alpha+2$ many variables, $c_i^{~,}s,$ which is equivalent to a matrix equation  of the form $Ax=0$ where  $A$ is an $(\alpha+3) \times (\alpha+2)$ matrix and  $x=[c_1,...,c_r]^T$:}

\[
\begin{bmatrix}
  
      q   \quad     & \frac{q}{2}~~   & \dots & ~~ \frac{q}{\alpha+1} & 1 \\ \\
    \cot(\frac{\pi }{q})\quad     &  \cot(\frac{2\pi }{q})~~ & \dots &   \cot(\frac{(\alpha+1)\pi }{q}) & 0  \\ \\
      1 & 1 & \dots & 1 & 0 \\ \\
   
     \cos(\frac{2\pi }{q}) \quad    &  \cos(\frac{4\pi }{q}) & \dots & \cos(\frac{2(\alpha+1)\pi }{q})  & 0\\ \\
     
      \cos(\frac{4\pi}{q})  \quad    &  \cos(\frac{8\pi }{q}) & \dots & \cos(\frac{4(\alpha+1)\pi }{q}) & 0  \\ \\
      \vdots & \vdots & \vdots & \vdots & \vdots\\ \\
    
     \cos(\frac{2\alpha \pi  }{q}) ~~    &  \cos(\frac{4 \alpha \pi}{q}) & \dots & \cos(\frac{2\alpha(\alpha+1) \pi }{q}) & 0  
\end{bmatrix}
\begin{bmatrix}{}
    ~c_1~ \\ \\
    c_2 \\ \\
    c_3 \\ \\
    \vdots \\ \\
     \vdots \\ \\
     \vdots \\ \\
    ~c_{\alpha+2}~
\end{bmatrix}
=
\begin{bmatrix}{}
    ~0 ~\\ \\
    ~0~ \\ \\
    ~0 ~\\ \\
    \vdots \\ \\
     \vdots \\ \\
     \vdots \\ \\
   ~ 0~
\end{bmatrix}
\]
Also observe that, $\cot \frac{ \alpha  \pi}{q}=\cot\frac{\pi}{2q}=\ - \cot \frac{ (\alpha+1) \pi}{q}$ and 

$$ \cos (\frac{2t\alpha \pi}{q})= \cos \frac{2t(\alpha+1)}{q} \hspace{.3cm} \text{ for all} \hspace{.2cm} 1 \leq t \leq \alpha.$$
Hence, from the  above identity,  the matrix $A$ satisfies
\begin{equation}\label{s}
A_{i(\alpha)} = A_{i (\alpha+1)} \hspace{.3cm} \text{ for}\hspace{.2cm}  4 \leq i \leq \alpha+3 .
\end{equation}
To show that the equation $Ax=0$ has a trivial solution, it is sufficient to find a non- singular $(\alpha+2) \times (\alpha+2) $ submatrix of 
$A.$  Let $B$ be the   $(\alpha+2) \times (\alpha+2)$ submatrix of the matrix $A$ obtained by removing the third row of the matrix $A.$
Thus, by using $\eqref{s}$  and the elementary column operations on the matrix $B,$ we get

\[
B=
\begin{bmatrix}

      q   \quad     & \frac{q}{2}~~   & \dots & ~~ \frac{q}{\alpha} &  ~~\frac{q}{\alpha+1}-\frac{q}{\alpha}&~~ 1 \\ \\
    \cot(\frac{\pi }{q})\quad     &  \cot(\frac{2\pi }{q})~~ & \dots &   \cot(\frac{\alpha\pi }{q}) & -2\cot(\frac{\alpha\pi }{q}) ~~& 0  \\ \\
   
     \cos(\frac{2\pi }{q}) \quad    &  \cos(\frac{4\pi }{q}) & \dots & \cos(\frac{2\alpha\pi }{q})  & 0 &0\\ \\
     
      \cos(\frac{4\pi}{q})  \quad    &  \cos(\frac{8\pi }{q}) & \dots & \cos(\frac{4\alpha\pi }{q}) & 0&0  \\ \\
      \vdots & \vdots & \vdots & \vdots & \vdots& \vdots \\ \\
    
     \cos(\frac{2\alpha \pi  }{q}) ~~    &  \cos(\frac{4 \alpha \pi}{q}) & \dots & \cos(\frac{2\alpha^2 \pi }{q}) & 0  &0
\end{bmatrix}
\]

 To show that the matrix $B$  is non-singular, it is sufficient to show that the submatrix 

\[
C=
\begin{bmatrix}

     \cos(\frac{2\pi }{q}) \quad    &  \cos(\frac{4\pi }{q}) & \dots & \cos(\frac{2\alpha\pi }{q})  \\ \\
     
      \cos(\frac{4\pi}{q})  \quad    &  \cos(\frac{8\pi }{q}) & \dots & \cos(\frac{4\alpha\pi }{q})  \\ \\
      \vdots & \vdots & \vdots & \vdots \\ \\
    
     \cos(\frac{2\alpha \pi  }{q}) ~~    &  \cos(\frac{4 \alpha \pi}{q}) & \dots & \cos(\frac{2\alpha^2 \pi }{q}) 
\end{bmatrix}
\]
is non-singular.
To prove that the above submatrix is non-singular, we will use the following idea from Galois theory.

 Note that for any $t\in \N,$ we have,  $2\cos \frac{2t\pi}{q}=\zeta_q^t + \zeta_q^{-t},$ where $\zeta_q $ is the primitive $q$-th root of unity.  
 For any odd prime $q,$ consider the maximal real subfield $\Q(\zeta_q + \zeta_q^{-1})$   of the $q$-th  cyclotomic field $ \Q(\zeta_q)$ . Also 
 we know that, $|\Q(\zeta_q + \zeta_q^{-1}): \Q|=\frac{q-1}{2}$  and the ring of algebraic integers is $\Z[\zeta_q + \zeta_q^{-1}].$ 

Consider the $$\text{ Gal }(\Q(\zeta_q + \zeta_q^{-1}):\Q)= \{ \sigma_a \hspace{.1cm}  | \hspace{.2cm} \sigma_a(\zeta_q + \zeta_q^{-1})= \zeta_q^a + \zeta_q^{-a}, 
\hspace{.2cm} 1 \leq a \leq \frac{q-1}{2}\}.$$
The conjugates of $\zeta_q + \zeta_q^{-1}$  forms an integral basis for $\Q(\zeta_q + \zeta_q^{-1}).$ Let $\omega_i= \zeta_q^i + 
\zeta_q^{-i} \text{ for } 1 \leq i \leq \frac{q-1}{2}.$ Hence, $$ C_{ij}=\frac{\sigma_i(\omega_j)}{2}, \hspace{.3cm} \text{for}\hspace{.2cm}  1 \leq i,j \leq \alpha.$$
Since   $\sigma_i^{~,}s$  are linearly independent over $\Q(\zeta_q + \zeta_q^{-1})$ by using Dedekind's  lemma and $\omega_i^{~,}s$ 
forms an integral basis, hence the matrix $C$ is non-singular.
Thus, we will get a trivial solution for  $c_i^{~,}s.$ This completes the proof.

\end{proof}

In our next theorem, we will consider the case for  any finite number of primes.

\begin{thm}
For any finite set of odd primes $J,$ with $|J|=n,$ define 
$$W_J=\overline{\Q}- \text {span of } \{ H_1, \  H_{a_{j_i}/q_i} | \ 1 \leq a_{j_i} \leq q_i -1, \ 1 \leq j_i \leq q_i-1, \ \ \forall q_i \in J\}.$$
Then, $$\text{ dim }_{\overline{\Q}} ~W_J=\sum\limits_{\substack{i=1 \\ q_i \in J}}^n \frac{\phi (q_i )}{2} + 2.$$
\end{thm}

\begin{proof}
First note that from theorem 3.10, we have
$$\text{ dim }_{\overline{\Q}} ~W_J\leq \sum\limits_{\substack{i=1 \\ q_i \in J}}^n \frac{\phi (q_i )}{2} + 3.$$
 Choose any subset of the set 
 $$ \{ H_1, \  H_{a_{j_i}/q_i} | \ 1 \leq a_{j_i} \leq q_i -1, \ 1 \leq j_i \leq q_i-1, \ \ \forall q_i \in J\}$$
 say $T,$
containing $\sum\limits_{i=1}^n \frac{\phi (q_i )}{2} + 3$ many elements.
 Consider the  equation
 \begin{equation}\label{zb}
 \sum\limits_{\substack{s \\ q_s \in J }}\sum_{t_s}  c_{st_s} H_{t_s/q_s^{m_s}}=0,
 \end{equation}
 where $c_{st_s}^{~,}s \in \overline{\Q}$ and  the sum runs over all $s,t_s$ such that $H_{t_s/q_s^{m_s}} \in T.$ Clearly the number of variables 
 $c_{st_s}^{~~,}s,$ is $\sum\limits_{i=1}^n \frac{\phi (q_i )}{2} + 3$.  Substituting the Gauss formula in \eqref{zb} and using lemma 3.6, 3.7, 
 Baker's theorem and proposition 2.2, we will get a system of linear homogeneous equations in $\sum\limits_{i=1}^n \frac{\phi (q_i )}{2} + 3$ 
 equations and $\sum\limits_{i=1}^n \frac{\phi (q_i )}{2} + 3$  variables.
Now  using  the similar technique as  we did in  remark 2 to prove that the determinant of the matrix is zero, it is not difficult to get that
$$\text{ dim }_{\overline{\Q}} ~W_J\leq \sum\limits_{i=1}^n \frac{\phi (q_i )}{2} + 2.$$
Now  consider the set
$$\{  H_{a_{j_i}/q_i}, \ H_{(\frac{q_1+1}{2})/q_1},\  H_1 \ | \  1 \leq a_{j_ i} \leq \frac{q_i-1}{2} \ \text{ and } 1 \leq j_i \leq \frac{q_i-1}{2},\  \forall \ q_i \in J\} $$
containing $\sum\limits_{i=1}^n \frac{\phi (q_i )}{2} + 2$ many elements.  Thus, to prove theorem 4.2, it is sufficient to show that the above set is linearly independent over $\overline{\Q}.$
Consider the $\overline{\Q}$- linear combination of the above set
$$\sum\limits_{i=1}^n\sum\limits_{j_i=1}^\frac{q_i-1}{2}c_{j_i}^i H_{a_{j_i}/q_i} + c_{(\frac{q_1+1}{2})/q_1} H_{(\frac{q_1+1}{2})/q_1} + c_1  H_1 =0. $$
 Substituting the value of $H_{a/q}$ in the above equation, we get

$$\bigg(\Big (\sum\limits_{i=1}^n\sum\limits_{j_i=1}^\frac{q_i -1}{2} c_{j_i}^i \frac{q_i}{{a_{j_i}}}\Big)  \ + \ c_{(\frac{q_1+1}{2})/q_1} \ 
\frac{2q_1}{q_1+1} \ +  \ c_1 \bigg) - \frac{\pi}{2}\bigg(\sum\limits_{i=1}^n \sum\limits_{i=1}^\frac{q_i-1}{2} (c_{j_i}^i \cot (\frac{{a_{j_i}} \pi}{q_i}))\bigg ) - $$

$$  - \frac{\pi}{2}\bigg( c_{(q_1+1)/2q_1} \cot(\frac{(q_1+1) \pi}{2q_1}\big) \bigg)  - \bigg( \sum\limits_{i=1}^n \sum\limits_{j_i =1}^\frac{q_i-1}{2}c_{j_i}^i 
\log(2q_i) \bigg) \ -\bigg( c_{(\frac{q_1+1}{2})/q_1}\log(2q_1) \bigg)  $$

  $$  + \bigg( 2 \sum\limits_{i=1}^n \sum\limits_{j_i=1}^\frac{q_i -1}{2}\bigg(\sum\limits_{k=1}^\frac{q_i-1}{2}c_{j_i}^i\cos ( \frac{2 a_{j_i} k \pi}{q_i})
  \log\sin(\frac{k\pi}{q_i})\bigg) \bigg) + $$
  
 \begin{equation}\label{al}
    + \bigg( 2 c_{(\frac{q_1+1}{2})/q_1}\bigg( \sum\limits_{k=1}^\frac{q_1 -1}{2}\cos(\frac{2(q_1+ 1) \pi k}{2q_1})\log \sin(\frac{k\pi}{q_1})\bigg) \bigg)=0.
 \end{equation} 

\textcolor{red}{Now proposition 2.2 implies both the algebraic term and the coefficient of $\pi$ in the  equation \eqref{al} to be simultaneously zero, that is}
   $$\Big(\sum\limits_{i=1}^n\sum\limits_{j_i=1}^\frac{q_i -1}{2} c_{j_i}^i \frac{q_i}{a_{j_i}}\Big)  \ + \ c_{(\frac{q_1+1}{2})/q_1} \  \frac{2q_1}{q_1+1} \ +  \ c_1 =0,$$
  $$
  - \frac{\pi}{2}\bigg(\Big(\sum\limits_{i=1}^n \sum\limits_{i=1}^\frac{q_i-1}{2} (c_{j_i}^i \cot (\frac{a_{j_i} \pi}{q_i})) \Big) \  + 
  \  c_{(q_1+1)/2q_1} \cot(\frac{(q_1+1) \pi}{2q_1}) \bigg) =0. $$
 
  \vspace{.4cm}
 Thus equation \eqref{al} reduces to

 $$- \Big( \ \sum\limits_{i=1}^n \Big( \sum\limits_{j_i =1}^\frac{q_i-1}{2}c_{j_i}^i \log(2q_i)\Big)\ \Big) - \ c_{(\frac{q_1+1}{2})/q_1}\Big(\log(2q_1)\Big) 
 $$
 $$+ \ \ 2 \bigg( \sum\limits_{i=1}^n \sum\limits_{j_i=1}^\frac{q_i -1}{2}c_{j_i}^i\bigg(  \sum\limits_{k=1}^\frac{q_i-1}{2}
 \Big(  \cos ( \frac{2 a_{j_i} k \pi}{q_i})\Big) \log\sin(\frac{k\pi}{q_i})\bigg) \bigg) $$ 

  \begin{equation}\label{p}
   +  \ \ 2 c_{(\frac{q_1+1}{2})/q_1} \bigg(\sum\limits_{k=1}^\frac{q_1 -1}{2}\Big(  \cos(\frac{2(q_1+ 1) \pi k}{2q_1})\Big) \log \sin(\frac{k\pi}{q_1})\bigg)=\ 0.
   \end{equation}
Also by using lemma 3.7, we have
$$\log q_i=2\Big(\log \sin(\frac{\pi}{q_i}) +... + \log \sin(\frac{\alpha_i\pi}{q_i}) + \alpha_i \log 2\Big)$$
where $ \alpha_i =\frac{q_i-1}{2}$ for all $ 1 \leq i \leq n.$ Substituting the value of $\log q_i $ in  \eqref{p}, we get,
$$- \Big(\Big(\ \sum\limits_{i=1}^n \sum\limits_{j_i =1}^\frac{q_i-1}{2}c_{j_i}^i  ( q_i ) \Big) \ \ + \ \ (c_{(\frac{q_1+1}{2})/q_1 })q_1\Big)\log2
  \ + \ 2 \sum\limits_{i=1}^n \sum\limits_{j_i=1}^\frac{q_i -1}{2}\sum\limits_{k=1}^\frac{q_i-1}{2}c_{j_i}^i\Big(  -1+\cos ( \frac{2 a_{j_i} k \pi}{q_i})\Big) \log\sin(\frac{k\pi}{q_i}) $$ 
  \begin{equation}\label{bb}
   +  \ 2 c_{(\frac{q_1+1}{2})/q_1} \sum\limits_{k=1}^\frac{q_1 -1}{2}\Big( -1+ \cos(\frac{2(q_1+ 1) \pi k}{2q_1})\Big) \log \sin(\frac{k\pi}{q_1})=0.
   \end{equation}

Now using the fact that the set
$$\{\log2,  \  \log \sin \frac{k_{j_i} \pi}{q_i } | \hspace{.2cm} 1 \leq k_{j_i}  \leq \frac{q_i-1}{2}  \text{ and }  1 \leq j_i \leq \frac{q_i -1}{2}, \ \forall \ q_i \in J \}$$
is linearly independent over $\overline{\Q}$ ( by lemma 3.6,  lemma 3.7 and Baker's theorem) in  \eqref{bb},
we get a linear homogeneous system of $\sum\limits_{i=1}^n \frac{\phi (q_i )}{2} + 1$  equations. Thus, altogether  we get a homogeneous system of 
$\sum\limits_{i=1}^n \frac{\phi (q_i )}{2} + 3 $ many equations in $\sum\limits_{i=1}^n \frac{\phi (q_i )}{2} + 2$  variables that is we will get a 
system of the form $Ax=0,$ where $A$ is a $(\sum\limits_{i=1}^n \frac{\phi (q_i )}{2} + 3) \times (\sum\limits_{i=1}^n \frac{\phi (q_i )}{2} + 2) $
matrix and $A$ is of the form: 

\[
A=
\begin{bmatrix}
  q_1  & \dots  & \frac{2q_1}{q_1 -1} & \dots   & \frac{2q_2}{q_2-1} & \dots & \frac{2q_n}{q_n-1} & \frac{2q_1}{q_1+1} - \frac{2q_1}{q_1-1} & 1 \\ \\
    q_1 & \dots  &  q_1 & \dots &  q_2 & \dots & q_n & 0 & 0 \\ \\
    \cot(\frac{\pi }{q_1})\quad     & \dots &  \cot(\frac{(q_1-1)\pi }{2q_1})  & \dots & \cot(\frac{(q_2-1)\pi }{2q_2} ) & \dots & 
    \vspace{.7cm}
     \cot(\frac{(q_n-1)\pi }{2q_n} & -2\cot(\frac{(q_1+1)\pi }{2q_1} & 0\\ \\
     \vspace{.4cm}
    &   {\huge A_{q_1}} & & &    O &    &   & 0  & 0 \\ \\
    & & & & &  & & \vdots & \vdots \\ \\
    & O & & &   A_{q_2} &   &  & 0 & 0 \\ \\
    & \vdots & & &  & \ddots  & &    \vdots & \vdots  \\ \\
    
     &  O &   &  & O & &  A_{q_n}  & 0 & 0\\ \\ 
\end{bmatrix}
\]
where $O$ is the zero matrix and  each submatrix $A_{q_r}$  is a $( \alpha_r \times \alpha_r)$ matrix where $ \alpha_r=\frac{q_r-1}{2}$ for all 
$ 1 \leq r \leq n.$  Also for any $ 1 \leq r \leq n,$  $A_{q_r}$ is of the form

\[
A_{q_r}=
\begin{bmatrix}

     -1 + \cos(\frac{2\pi }{q_r}) \quad    &  -1 + \cos(\frac{4\pi }{q_r}) & \dots & -1 + \cos(\frac{2\alpha_r\pi }{q_r})  \\ \\
     
     -1 +  \cos(\frac{4\pi}{q_r})  \quad    &  -1+ \cos(\frac{8\pi }{q_r}) & \dots & -1 +  \cos(\frac{4\alpha_r\pi }{q_r})  \\ \\
      \vdots & \vdots & \vdots & \vdots \\ \\
    
   -1+ \cos(\frac{2\alpha_r \pi  }{q_r}) ~~    & -1 +  \cos(\frac{4 \alpha_r \pi}{q_r}) & \dots & -1 + \cos(\frac{2\alpha_r^2 \pi }{q_r}) 
\end{bmatrix}
\]
 for all $ 1 \leq r \leq n.$\\
Now to show that the equation $Ax=0$ has only trivial solution, it is sufficient to show that the matrix $A$ has rank $\sum\limits_{i=1}^n \frac{\phi (q_i )}{2} + 2$.
Consider the submatrix of $A,$  say $B$ after removing the second row, which is an $\big( \sum\limits_{i=1}^n \frac{\phi (q_i )}{2} + 2\big)  
\times \big(\sum\limits_{i=1}^n \frac{\phi (q_i )}{2} + 2\big)$ matrix. Now to show that the submatrix $B$ is non-singular, it is sufficient to 
show that each submatrix $ A_{q_r}$ for all $ 1\leq r \leq n$ is non-singular.

 Note that for any $t\in \N,$ we have,  $2\cos \frac{2t\pi}{q_r}=\zeta_{q_r}^t + \zeta_{q_r}^{-t},$ where $\zeta_{q_r} $ is the primitive $q_r$-th root of unity. 
 Also we know that, $|\Q(\zeta_{q_r} + \zeta_{q_r}^{-1}): \Q|=\frac{q_r-1}{2}$ for all $ 1 \leq r \leq n.$ 
Consider the $$\text{ Gal }(\Q(\zeta_{q_r} + \zeta_{q_r}^{-1}):\Q)= \{ \sigma_a \hspace{.1cm}  | \hspace{.2cm} \sigma_a(-2 + \zeta_{q_r} + \zeta_{q_r}^{-1})= -2 
+ \zeta_{q_r}^a + \zeta_{q_r}^{-a}, \hspace{.2cm} 1 \leq a \leq \frac{q_r-1}{2}\}$$
for all $1 \leq r \leq n.$
The conjugates of $-2 + \zeta_{q_r} + \zeta_{q_r}^{-1}$  forms an integral basis for $\Q(\zeta_{q_r} + \zeta_{q_r}^{-1})$ for all $ 1 \leq r \leq n$. 
Let $\omega_{a_r}^r= -2 + \zeta_{q_r}^{a_r}+ \zeta_{q_r}^{-a_r} \text{ for } 1 \leq a_r \leq \frac{q_r-1}{2}$ for all $ 1 \leq r \leq n$.  
Then for each $ 1 \leq r \leq n,$ \  $ \{\omega_{a_r}^r \ | \ 1 \leq a_r \leq \frac{q_r-1}{2}\} $ forms an integral basis for $\Q(\zeta_{q_r} + \zeta_{q_r}^{-1})$.
Also observe that $$ (A_{q_r})_{ij}=\frac{\sigma_i(\omega_j^r)}{2}, \hspace{.3cm} \text{for}\hspace{.2cm}  1 \leq i,j \leq \frac{q_r-1}{2}.$$
Since   $\sigma_i^{~,}s$  are linearly independent over $ \Q(\zeta_{q_r} + \zeta_{q_r}^{-1})$  for all $ 1 \leq r \leq n$ by using Dedekind's lemma 
and ${\omega_{j}^r}^{~,}s$ forms an integral basis, for each $ 1 \leq r \leq n,$ thus the matrix $B$ is non-singular.
Hence, we will get a trivial solution for  $c_{j_i}^{~,}s.$ This completes the proof.
\end{proof}

\par
\noindent
{\bf Acknowledgements.}
The authors thank  Prof. M. Ram Murty for his comments on an earlier version of this paper. The authors would like to thank the referee for his insightful
suggestions. The authors also thank Mr. Suraj Singh Khurana for his help in the revision process of this article. The first author
would like to thank Universit\'{e} Pierre et Marie Curie (University of Paris VI) for the hospitality where some part of the work
was done and the National Board for Higher Mathematics for the partial support.


\begin{thebibliography}{10}


\bibitem{AB}
A. Baker, {\em Transcendental Number Theory}, Cambridge univ. Press,  (1975).

\bibitem{bass}
H. Bass, \textit{Generators and relations for cyclotomic units}, Nagoya Math Journal, Volume 27, Part 2 (1966), 401-407.

\bibitem{CD}
T. Chatterjee and S. Dhillon, {\em Linear independence of  harmonic numbers over the field of algebraic numbers II}, preprint.

\bibitem{TCS}
T. Chatterjee and S. Gun, {\em The digamma function, Euler-Lehmer constants and their p-adic counterparts}, Acta Arithmetica
{\bf162 (2)} (2014) 197--208, MR3167891.

\bibitem{CK}
T. Chatterjee and S. S. Khurana, {\em Erd\H{o}sian functions and Associated L-functions}, submitted.

\bibitem{CM}

T. Chatterjee and M. R. Murty, {\em Non-vanishing of Dirichlet series with periodic coefficients},
J. Number Theory, {\bf145} (2014), 1--21, MR3253290.

\bibitem{CMP}
 T. Chatterjee, M. Ram Murty and Siddhi Pathak, {\em A vanishing criterion for Dirichlet series with periodic coefficients},
Contemporary Mathematics, AMS Proceedings, 701, 2018, 69--80.


\bibitem{ennola}
V. Ennola, \textit{On relations between Cyclotomic Units}, Journal of Number Theory, \textbf{4}, (1972) 236-247.


\bibitem{MM}

M. Ram Murty, V. Kumar Murty, {\em A problem of Chowla revisited}, J. Number Theory {\bf131 (9)} (2011) 1723-1733, MR2802143.



\bibitem{MS1}

M. Ram Murty and N. Saradha, {\em Transcendental values of the digamma function}, J. Number Theory {\bf 125} (2007) 298--318.

\bibitem{RM}

M. Ram  Murty and N. Saradha, {\em  Euler-Lehmer constants and a conjecture of Erdös}, J. Number Theory  {\bf 130 (12)}  (2010) 2671-
2682.


\end{thebibliography}
\end{document}